\documentclass[reqno]{amsart}
\usepackage[utf8]{inputenc}
\usepackage[letterpaper, portrait, margin=1.2in]{geometry}
\usepackage[english]{babel}
\usepackage{amsthm}
\usepackage[utf8]{inputenc}
\usepackage{graphicx}
\usepackage{mathtools}
\usepackage{amsmath}
\usepackage{amssymb}
\usepackage{tabto}
\usepackage{enumitem}
\usepackage{hyperref}
\usepackage{graphicx,subcaption}
\usepackage{tikz}
\usetikzlibrary{calc}
\usetikzlibrary{backgrounds,decorations}
\usepackage{xstring}
\newtheorem*{theorem*}{Theorem}
\newtheorem{theorem}{Theorem}[section]
\newtheorem{corollary}[theorem]{Corollary}
\newtheorem{problem}[theorem]{Problem}
\newtheorem{lemma}[theorem]{Lemma}

\theoremstyle{definition}
\newtheorem{definition}[theorem]{Definition}
\newtheorem{example}[theorem]{Example}

\DeclareMathOperator{\Id}{\mathrm{Id}}
\DeclareMathOperator{\rev}{\mathrm{rev}}
\usepackage[skip=0.5\baselineskip]{caption}
\usepackage[toc,page]{appendix}

\newcommand{\NEpath}[4]{
    \fill[white!25]  (#1) rectangle +(#2,#3);
    \fill[fill=white]
    (#1)
    \foreach \dir in {#4}{
        \ifnum\dir=0
        -- ++(1,0)
        \else
        -- ++(0,1)
        \fi
    } |- (#1);
    \draw[help lines] (#1) grid +(#2,#3);
    \draw[dashed] (#1) -- +(#3,#3);
    \coordinate (prev) at (#1);
    \foreach \dir in {#4}{
        \ifnum\dir=0
        \coordinate (dep) at (1,0);
        \else
        \coordinate (dep) at (0,1);
        \fi
        \draw[line width=2pt,-stealth] (prev) -- ++(dep) coordinate (prev);
    };
}

\newcommand{\drawpermutate}[1]{%
    \begin{tikzpicture}
        \foreach [evaluate=\i as \x using int(\i-1)]\i in {0,1,...,8}
        {
            \foreach [evaluate=\j as \y using int(\j-1)] \j in {0,1,...,8}
            {
                \node at (\i,\j)[name=perm-\x-\y,]{};
            }
        }

        \foreach \i in {1,...,6}
        {
            \node at (perm-\i-0.center){\i};
            \node at (perm-0-\i.center){\i};
        }

        \draw ([xshift=-5mm]perm-1--1.south west)--([xshift=-5mm]perm-1-7.north west);
        \draw ([yshift=5mm]perm--1-0.south west)--([yshift=5mm]perm-7-0.south east);

        \foreach \mystyle/\coords/\leftorright in {#1}
        {
            \foreach \x/\y in \coords
            {
                \node[circle,fill=\mystyle,draw=\mystyle] at (perm-\x-\y){};
            }

        }

    \end{tikzpicture}
}

\newcommand{\drawperms}[1]{%
    \begin{tikzpicture}
        \foreach [evaluate=\i as \x using int(\i-1)]\i in {0,1,...,8}
        {
            \foreach [evaluate=\j as \y using int(\j-1)] \j in {0,1,...,8}
            {
                \node at (\i,\j)[name=perm-\x-\y,]{};
            }
        }

        \draw ([xshift=-5mm]perm-1--1.south west)--([xshift=-5mm]perm-1-7.north west);
        \draw ([yshift=5mm]perm--1-0.south west)--([yshift=5mm]perm-7-0.south east);

        \foreach \mystyle/\coords/\leftorright in {#1}
        {
            \foreach \x/\y in \coords
            {
                \node[circle,fill=\mystyle,draw=\mystyle] at (perm-\x-\y){};
            }

        }

    \end{tikzpicture}
}

\title{A Complete Enumeration of Ballot Permutations Avoiding Sets of Small Patterns}
\author{Nathan Sun} \address{Department of Applied Mathematics, Harvard University, Cambridge, MA 02138} \email{nsun@college.harvard.edu}
\date{June 26, 2022}
\subjclass[2020]{Primary: 05A05}
\keywords{Permutations, ballot permutations, pattern avoidance, Dyck paths}

\begin{document}
\maketitle

\begin{abstract}
Permutations whose prefixes contain at least as many ascents as descents are called ballot permutations. Lin, Wang, and Zhao have previously enumerated ballot permutations avoiding small patterns and have proposed the problem of enumerating ballot permutations avoiding a pair of permutations of length $3$. We completely enumerate ballot permutations avoiding two patterns of length $3$ and we relate these avoidance classes with their respective recurrence relations and formulas, which leads to an interesting bijection between ballot permutations avoiding $132$ and $312$ with left factors of Dyck paths. In addition, we also conclude the Wilf-classification of ballot permutations avoiding sets of two patterns of length $3$, and we then extend our results to completely enumerate ballot permutations avoiding three patterns of length $3$.
\end{abstract}

\section{Introduction}

The distribution of descents over permutations has been thoroughly researched and has several important combinatorial properties. Specifically, the Eulerian polynomials $A_n(t)$ encapsulate information about the number of descents in every permutation in $S_n$, and $q$-analogues defined using additional permutation statistics have been considered by Agrawal, Choi, and Sun \cite{agrawal2020permutation}, Carlitz \cite{carlitz1954q}, and Foata and Sch{\"u}tzenberger \cite{foata1978major}. In particular, the Eulerian polynomials can also be equivalently defined using the excedance permutation statistic. Spiro \cite{spiro2020ballot} introduced a variation of this in his work on ballot permutations, of which we will now give a brief history.

The following ballot problem was first introduced by Bertrand \cite{bertrand} in 1887 for the case $\lambda = 1$. 
\begin{problem}
Suppose in an election, candidate A receives $a$ votes and candidate B receives $b$ votes, where $a \geq \lambda b$ for some positive integer $\lambda$. How many ways can the ballots in the election be ordered such that candidate A maintains more than $\lambda$ times as many votes as candidate B throughout the counting of the ballots?
\end{problem}

Almost immediately after Bertrand introduced the ballot problem, André \cite{andre} introduced ballot sequences in a combinatorial solution, and more recently, Goulden and Serrano \cite{goulden} provided a solution to the case where $\lambda>1$ using a variation of ballot sequences. However, the most famous variation of ballot sequences is ballot permutations, which represent each vote for candidate A and candidate B via an ascent and a descent in the permutation, respectively. Ballot permutations have been studied by Spiro \cite{spiro2020ballot}, Bernardi, Duplantier, and Nadeau \cite{bernardi2010bijection}, and Lin, Wang, and Zhao \cite{lin2022decomposition}. In particular, Bernardi, Duplantier, and Nadeau \cite{bernardi2010bijection} proved that the set of ballot permutations with length $n$ are equinumerous to the set of odd order permutations with the same length. Spiro \cite{spiro2020ballot} introduced a variation of excedence numbers, whose distribution over the set of odd order permutations is the same as the distribution of the descent numbers over the set of ballot permutations.

In an extension to Spiro's \cite{spiro2020ballot} work, Lin, Wang, and Zhao \cite{lin2022decomposition} constructed an explicit bijection between these two sets of permutations, which can be extended to positive well-labeled paths and proves a conjecture due to Spiro \cite{spiro2020ballot} using the statistic of peak values. Lin, Wang, and Zhao \cite{lin2022decomposition} also established a connection between $213$-avoiding ballot permutations and Gessel walks and initiated the enumeration of ballot permutations avoiding a single pattern of length $3$. They have also suggested the problem of enumerating ballot permutations avoiding pairs of permutations of length $3$, on which we will now present two main results. We first completely enumerate ballot permutations avoiding two patterns of length $3$ and prove their respective recurrence relations and formulas. In doing this, we characterize the set of ballot permutations avoiding sets of patterns. We then show a bijection between $132$- and $213$-avoiding ballot permutations with left factors of Dyck paths and establish all  Wilf-equivalences between patterns. We finally initiate and completely enumerate ballot permutations avoiding three patterns of length $3$. 

This paper is organized as follows. In Section 2, we introduce preliminary definitions and notation. In Section 3, we completely enumerate ballot permutations avoiding two patterns of length $3$ and prove their respective recurrence relations and formulas. In addition, we prove Wilf-equivalences of patterns and show a bijection to left factors of Dyck paths. In Section 4, we extend our enumeration to ballot permutations avoiding three patterns of length $3$. In Section 5, we conclude with open problems and further directions.

\section{Preliminaries} \label{sec:preliminaries}
The following notation is borrowed from \cite{sun2022d}. We write $S_n$ to denote the set of permutations of $[n] = \{ 1,2, \dots, n \}$. Note that we can represent each permutation $\sigma \in S_n$ as a sequence $\sigma(1) \cdots \sigma(n)$. Further, let $\mathrm{Id}_n$ denote the identity permutation $12 \cdots n$ of size $n$ and given a permutation $\sigma \in S_n$, let $\rev(\sigma)$ denote the reverse permutation $\sigma(n) \sigma(n-1) \cdots \sigma(1)$. We further say that a sequence $w$ is \emph{consecutively increasing} (respectively \emph{decreasing}) if for every index $i$, $w(i+1) = w(i)+1$ (respectively $w(i+1) = w(i)-1$). 

For a sequence $w = w(1) \cdots w(n)$ with distinct real values, the \emph{standardization} of $w$ is the unique permutation with the same relative order. Note that once standardized, a consecutively-increasing sequence is the identity permutation and a consecutively-decreasing sequence is the reverse identity permutation. Moreover, we say that in a permutation $\sigma$, the elements $\sigma(i)$ and $\sigma(i+1)$ are \emph{adjacent} to each other. More specifically, $\sigma(i)$ is \emph{left-adjacent} to $\sigma(i+1)$ and similarly, the element $\sigma(i+1)$ is $\emph{right-adjacent}$ to $\sigma(i)$. The definitions in this section are taken from \cite{lin2022decomposition}.

\begin{definition}
A \emph{prefix} of a permutation $\sigma$ is a contiguous initial subsequence $\sigma(1) \cdots \sigma(p)$ for some $p$.
\end{definition}

\begin{definition}
Given a permutation $\sigma \in S_n$, we say that $i \in [n-1]$ is a \emph{descent} of $\sigma$ if $\sigma(i)>\sigma(i+1)$. Similarly, we say that $i \in [n-1]$ is an \emph{ascent} of $\sigma$ if $\sigma(i) < \sigma(i+1)$.
\end{definition}

\begin{definition}
A \emph{ballot permutation} is a permutation $\sigma$ such that any prefix of $\sigma$ has at least as many ascents as descents.
\end{definition}

We let $B_n$ to denote the set of all ballot permutations of length $n$. It is interesting to consider the notion of pattern avoidance on ballot permutations, which we will now introduce.

\begin{definition}
We say that the permutation $\sigma$ \emph{contains} the permutation $\pi$ if there exists indices $c_1 < \dots < c_k$ such that $\sigma(c_1) \cdots \sigma(c_k)$ is order-isomorphic to $\pi$. We say that a permutation \emph{avoids} a pattern $\pi$ if it does not contain it. 
\end{definition}

Given patterns $\pi_1, \dots, \pi_m$, we let $B_n(\pi_1, \dots, \pi_m)$ to denote the set of all ballot permutations of length $n$ that avoid the patterns $\pi_1, \dots, \pi_n$. 

\begin{definition}
We say that two sets of patterns $\pi_1, \dots, \pi_k$ and $\tau_1, \dots, \tau_\ell$ are \emph{Wilf-equivalent} if $|S_n(\pi_1, \dots, \pi_k)| = |S_n(\tau_1, \dots, \tau_\ell)|.$ In the context of ballot permutations, we say that these two sets of patterns are Wilf-equivalent if $|B_n(\pi_1, \dots, \pi_k)| = |B_n(\tau_1, \dots, \tau_\ell)|.$
\end{definition}

To characterize permutations, we will now define the direct sum and the skew sum of permutations.

\begin{definition}
Let $\sigma$ be a permutation of length $n$ and $\pi$ be a permutation of length $m$. Then the \emph{skew sum} of $\sigma$ and $\pi$, denoted $\sigma \ominus \pi$, is defined by \begin{equation*}
    \sigma \ominus \pi (i) = \begin{cases}
\sigma(i) + m & 1 \leq i \leq n \\
\pi(i-n) & n+1 \leq i \leq m+n.
\end{cases}
\end{equation*}
\end{definition}

\begin{example}\label{skewsum}
As illustrated in Figure \ref{box}, $$132 \ominus 123 = 465123.$$
\end{example}

\begin{figure}[htp]
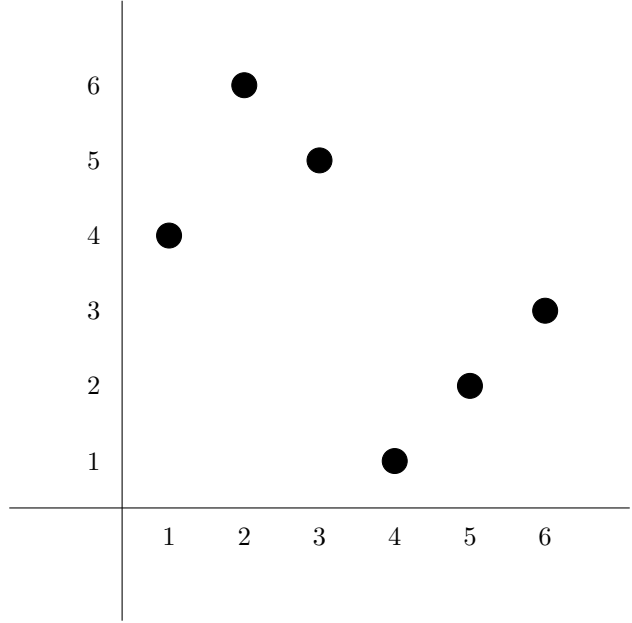

    \centering
    \drawpermutate{black/{1/4,3/5,5/2}/left,black/{4/1,2/6,6/3}/right}
    \caption{The graph of the skew sum described in Example \ref{skewsum}.}
    \label{box}
\end{figure}

\begin{definition}
Let $\sigma$ be a permutation of length $n$ and $\pi$ be a permutation of length $m$. Then the \emph{direct sum} of $\sigma$ and $\pi$, denoted $\sigma \oplus \pi$, is defined by \begin{equation*}
    \sigma \oplus \pi (i) = \begin{cases}
\sigma(i) & 1 \leq i \leq n \\
\pi(i-n)+n & n+1 \leq i \leq m+n.
\end{cases}
\end{equation*}
\end{definition}

\begin{example}\label{directsum}
As illustrated in Figure \ref{box2}, $$132 \oplus 123 = 132456.$$
\end{example}

\begin{figure}[htp]
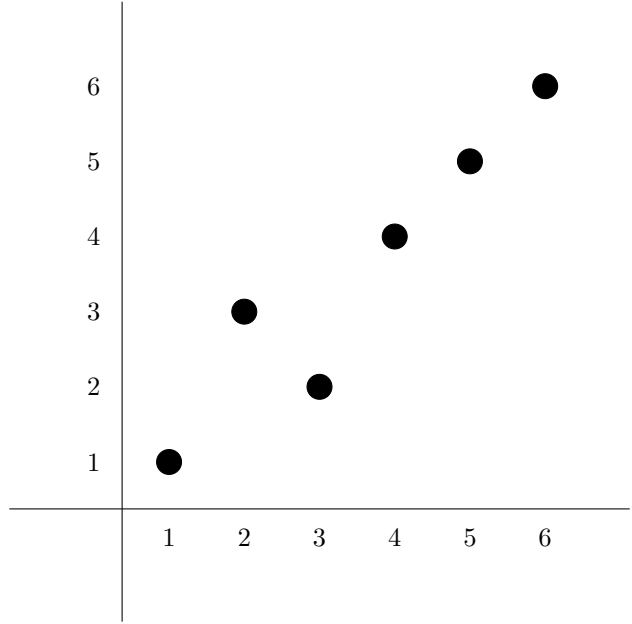

    \centering
    \drawpermutate{black/{1/1,3/2,5/5}/left,black/{4/4,2/3,6/6}/right}
    \caption{The graph of the direct sum described in Example \ref{directsum}.}
    \label{box2}
\end{figure}

\section{Enumeration of Pattern Avoidance Classes of length 2}
\label{sec:enumeration}

Lin, Wang, and Zhao \cite{lin2022decomposition} have enumerated sequences of ballot permutations avoiding small patterns. They provide the following Table \ref{double avoidance}:

\begin{table}[htp]
    \centering
    \begin{tabular}{|c | c | c | c|} 
 \hline
 Patterns & Sequence & OEIS Sequence & Comment \\ [0.5ex] 
 \hline\hline
 $123$ & $1,1,2,2,5,5,14,14, \dots$  & \href{http://oeis.org/A208355}{A208355} & Catalan number $C(\lceil \frac{n}{2} \rceil)$ \\ 
 \hline
 $132$ & $1,1,2,4,10,25,70, \dots$ & \href{https://oeis.org/A005817}{A005817} & $C(\lceil \frac{n}{2} \rceil)C(\lceil \frac{n+1}{2} \rceil)$ \\
 \hline
 $213$ & $1,1,3,6,21,52,193, \dots$ & \href{https://oeis.org/A151396}{A151396} & Gessel walks ending on the $y$-axis \\
 \hline
 $231$ & $1,1,2,4,10,25,70, \dots$ & \href{https://oeis.org/A005817}{A005817} & Wilf-equivalent to pattern $132$ \\
 \hline
 $312$ & $1,1,3,6,21,52,193, \dots$ & \href{https://oeis.org/A151396}{A151396} & Wilf-equivalent to pattern $213$ \\
 \hline
 $321$ & $1,1,3,9,28,90,297, \dots$ & \href{https://oeis.org/A071724}{A071724} & $\frac{3}{n+1} {{2n-2} \choose {n-2}}$ for $n >1$ \\
 \hline
\end{tabular}
    \caption{Sequences of ballot permutations avoiding one pattern of length $3$.}
    \label{double avoidance}
\end{table}

We extend Lin, Wang, and Zhao's \cite{lin2022decomposition} results to enumerate ballot permutations avoiding two patterns of length $3$. Table \ref{double avoidance} presents the sequence of ballot permutations avoiding two patterns of length $3$.

\begin{table}[htp]
    \centering
    \begin{tabular}{|c | c | c | c | c|} 
 \hline
 Patterns & Sequence & OEIS Sequence & Comment \\ [0.5ex] 
 \hline\hline
 $123, 132$ & $1,1,1,1,1,1,1 \dots$  & & Sequence of all $1$s; Theorem \ref{123,132} \\
 \hline
 $123, 213$ & $1,1,2,1,2,1,2, \dots$ & & Excluding $n=1$; Theorem \ref{123,213}  \\
 \hline
 $123, 231$ & $1,1,1,0,0,0,0, \dots$ & & Terminates after $n=3$  \\
 \hline
 $123, 312$ & $1,1, 2,0,0,0,0, \dots$ & & Terminates after $n=3$ \\
 \hline
 $123, 321$ & $1,1,2,2,0,0,0, \dots$ &  & Terminates after $n=4$ \\
 \hline
 $132, 213$ & $1,1,2,3,6,10,20, \dots$ & \href{https://oeis.org/A001405}{A001405} & Theorem \ref{132,213} \\
 \hline
 $132, 231$ & $1,1,1,1,1,1,1, \dots$ &  & Sequence of all $1$s; Theorem \ref{132,231} \\
 \hline
 $132, 312$ & $1,1,2,3,6,10,20, \dots$ & \href{https://oeis.org/A001405}{A001405} & Wilf-equivalent to patterns $132,213$ \\ 
 \hline
 $132, 321$ & $1,1,2,4,7,11,16, \dots$ & \href{https://oeis.org/A152947}{A152947} & Theorem \ref{132,321} \\ 
 \hline
 $213, 231$ & $1,1,2,3,6,10,20, \dots$ & \href{https://oeis.org/A001405}{A001405} & Wilf-equivalent to patterns $132,213$ \\ 
 \hline
 $213, 312$ & $1,1,3,4,11,16,42, \dots$ & \href{https://oeis.org/A027306}{A027306} & Theorem \ref{213,312} \\ 
 \hline
 $213, 321$ & $1,1,3,6,10,15,21, \dots$ & \href{https://oeis.org/A000217}{A000217} & Excluding $n=1$; Theorem \ref{213,321} \\ \hline
 $231, 312$ & $1,1,2,3,6,10,20, \dots$ & \href{https://oeis.org/A001405}{A001405} & Wilf-equivalent to patterns $132,213$ \\
 \hline
 $231, 321$ & $1,1,2,4,8,16,32, \dots$ & \href{https://oeis.org/A011782}{A011782} & Theorem \ref{231,321} \\
  \hline
 $312, 321$ & $1,1,3,6,12,24,48, \dots$ & \href{https://oeis.org/A003945}{A003945} & Excluding $n=1$; Theorem \ref{312,321} \\
 \hline
\end{tabular}
    \caption{Sequences of ballot permutations avoiding two patterns of length 3.}
    \label{double avoidance}
\end{table}

We first present a lemma, which will be used in the proofs of Theorems \ref{123,132} and \ref{123,213}.

\begin{lemma}\label{123 characterization}
Let $\sigma \in B_n(123)$, where $n$ is odd. Then either $\sigma(n)=1$ or $\sigma(n-2)=1$.
\end{lemma}
\begin{proof}
 Write $\sigma = \sigma_L 1 \sigma_R $ and let $\sigma$ be a ballot permutation avoiding $123$. Since $\sigma$ avoids the pattern $123$, it cannot have two consecutive ascents. But because $\sigma$ is also a ballot permutation, it has at least as many ascents as descents, and we conclude that $\sigma$ has an equal number of ascents and descents. Now if $\sigma_R$ is empty, then $\sigma(n)=1$. If $\sigma_R$ is nonempty, $\sigma_R$ must be decreasing to avoid the pattern $123$. Also $\sigma_R$ must be at most $2$ elements, since if it contained more, then either it would contain an instance of $123$ or there would be more descents than ascents in $\sigma$. But $\sigma_R$ cannot be one element, or else it would end in an ascent, which is impossible since $\sigma$ begins with an ascent and $n$ is odd. Hence $\sigma_R$ must be $2$ elements if it is nonempty, and thus $\sigma(n-2)=1$. 
\end{proof}

Now we proceed to enumerate ballot permutations avoiding pairs of patterns. We first consider when one of patterns is $123$.

\begin{theorem}\label{123,132}
For all $n$, there exists a unique ballot permutation avoiding the patterns $123$ and $132$.
\end{theorem}
\begin{proof}
Let $\sigma \in B_n(123,132)$ and write $\sigma = \sigma_L 1 \sigma_R$. Note that the case where $n=2$ is immediate, so for the following proof, assume $n>2$. We have two cases:

\begin{enumerate}
    \item $n$ is even.
    
    Using the same logic as in Lemma \ref{123 characterization}, we conclude that $\sigma$ has one more ascent than descent. So $\sigma_R$ cannot be empty and must only be one element to simultaneously avoid $123$ and $132$.
    
    We claim that $\sigma_R$ must be $2$ (the second minimal element in $\sigma$). For the sake of contradiction, suppose that $\sigma_R = r >2$. If there exists an element $m>2$ between $2$ and $1$ in $\sigma$, then $2mr$ is a subsequence of $\sigma$ and is an occurrence of $132$ or $123$. If there does not exist such an element $m$ between $2$ and $1$, then they are adjacent, and hence $\sigma$ contains two consecutive descents and hence is not a ballot permutation. Thus $\sigma_R = 2$.
    
    Note that $\sigma = \sigma_L 12$. Then $\sigma_L$ is a prefix of $\sigma$ and therefore is in $B_{n-2}(123,132)$, so we can inductively use the above reasoning to conclude that $\sigma_L = (12) \ominus (12) \ominus \dots \ominus (12)$. Hence there is a unique $\sigma$ in $B_n(123,132)$.
    
    \item $n$ is odd.
    
    Using Lemma \ref{123 characterization}, $\sigma_R$ must either be empty or two elements. But if $\sigma_R$ contains two elements, it either contains $123$ or $132$. So we conclude that $\sigma_R$ must be empty and hence $\sigma = \sigma_L 1$. And because $\sigma_L$ is a prefix of $\sigma$, it must be in $B_{n-1}(123,132)$, and note that $\sigma_L = (12) \ominus (12) \ominus \dots \ominus (12)$ follows immediately from Case 1. Hence there is a unique $\sigma$ in $B_n(123,132)$.
\end{enumerate}

Thus there is a unique ballot permutation avoiding the patterns $123$ and $132$.
\end{proof}

\begin{theorem}\label{123,213}
Let $a_n = |B_n(123,213)|$. Then $$a_n = \begin{cases}
1, & n=1 \text{ or n is even} \\
2, & \text{ otherwise}.
\end{cases}$$
\end{theorem}
\begin{proof}
Let $\sigma \in B_n(123,213)$ and write $\sigma = \sigma_L 1 \sigma_R$. We have two cases:

\begin{enumerate}
    \item $n$ is even.
    
    Then following the same reasoning in Theorem \ref{123,132}, $\sigma$ contains one more ascent than descent and hence $\sigma_R$ cannot be empty; namely $\sigma_R = 2$. So $\sigma_L$ is a prefix of $\sigma$ and therefore is in $B_{n-2}(123,213)$. A similar inductive argument presented in Theorem \ref{123,132} above shows that $\sigma_L = (12) \ominus (12) \ominus \dots \ominus (12)$. Hence only $\sigma$ is in $B_n(123,213)$.
    
    \item $n$ is odd.
    
    Then from Lemma \ref{123 characterization}, either $\sigma = \sigma_L 1$ or $\sigma = \sigma_L 1 a b$.
    
    If $\sigma = \sigma_L 1$, then the same argument in Theorem \ref{123,132} concludes that $\sigma_L = (12) \ominus \dots \ominus (12)$. 
    
    If $\sigma = \sigma_L 1 a b$, then $ab$ must be $32$ to avoid $123$ and $213$. Then $\sigma_L = (12) \ominus \dots \ominus (12)$ follows by the same argument.
    
    If $n$ is odd, there are two different elements in $B_n(123,213)$.
\end{enumerate}

Therefore, $a_n = 2$ if $n$ is odd and $a_n = 1$ if $n$ is even.
\end{proof}

We will now show four sets of patterns are Wilf-equivalent to each other via bijection.

\begin{theorem}\label{wilf equivalence}
The four sets $B_n(132,213)$, $B_n(213,231)$, $B_n(231,312)$, and $B_n(132,312)$ are Wilf-equivalent.
\end{theorem}
\begin{proof}
In each of the following bijections between $B_n(\pi_1, \pi_2)$ and $B_n(\pi_1', \pi_2')$, we first construct a bijection from $S_n(\pi_1, \pi_2)$ to $S_n(\pi_1', \pi_2')$ that preserves the positions of each descent and ascent in every permutation, and hence this restricts the bijection from $B_n(\pi_1, \pi_2)$ to $B_n(\pi_1', \pi_2')$.

We first show a bijection between $B_n(132,213)$ and $B_n(213,231)$. Elements in $B_n(132,213)$ are of the form $\Id_{k_1} \ominus \Id_{k_2} \ominus \dots \ominus \Id_{k_m}$, as illustrated in Figure \ref{Wilfs}. Note that $k_1 >1$ while any other $k_i$ (where $i \neq 1$) may equal $1$, as long as the resulting permutation is a ballot permutation. We must have $k_1 > 1$ since the permutation is a ballot permutation and must start with an ascent. The permutation must be in this form since ascents must be consecutive to avoid $132$ and if there is a descent between element $i$ and element $j$, then consecutive ascents after $j$ must cover all elements up to $i$ to avoid $213$. 

Similarly, elements in $B_n(213,231)$ can be written as $((\Id_{k_1'} n \oplus \Id_{k_2'}) (n-1) \oplus \dots \oplus \Id_{k_m'}) (n-m+1)$, where the $(n-i+1)$ terms do not change under direct sum and each $(n-i+1)$ term is the largest element of every element after it. In other words, these terms are essentially ignored in the direct sum operations. This is also shown in Figure \ref{Wilfs}. Note that $k_1'>0$ while any other $k_i'$ may equal $0$, as long as the resulting permutation is a ballot permutation. Now we can rewrite this as $\sigma_{k_1} \oplus \sigma_{k_2} \oplus \dots \oplus \sigma_{k_m}$, where $\sigma_{k_i} = \Id_{k_i'} (n-i+1)$ and the $(n-i+1)$ terms does not change under direct sum.

And hence we can send $\sigma_{k_1} \oplus \sigma_{k_2} \oplus \dots \oplus \sigma_{k_m}$ to $\Id_{k_1} \ominus \Id_{k_2} \ominus \dots \ominus \Id_{k_m}$, due to each $\sigma_{k_i}$ ending in $(n-i-1)$. Note that this is a bijection from $S_n(132,213)$ to $S_n(213,231)$ that preserves the positions of ascents and descents in every permutation, so this property restricts the bijection between $B_n(132,213)$ to $B_n(213,231)$.

Now we show a bijection between $B_n(213,231)$ and $B_n(231,312)$. As discussed above, elements in $B_n(213,231)$ can be written in the form $\sigma_{k_1} \oplus \sigma_{k_2} \oplus \dots \oplus \sigma_{k_m}$, where $\sigma_{k_i} = \Id_{k_i'} (n-i+1)$. Now note that elements in $B_n(231,312)$ are in the form of $1 \oplus \rev(\Id_{k_1}) \oplus \dots \oplus \rev(\Id_{k_m})$, where each $\Id_{k_i}$ may be one element. Elements in $B_n(213,231)$ are of the form $((\Id_{k_1'} n \oplus \Id_{k_2'}) (n-1) \oplus \dots \oplus \Id_{k_m'}) (n-m+1)$. Note that $k_1'>0$ while any other $k_i'$ may equal $0$. Now we transform $1 \oplus \rev(\Id_{k_1}) \oplus \dots \oplus \rev(\Id_{k_m})$ into an element in $B_n(213,231)$ by preserving the place of each ascent and descent. This expression can be rewritten as the direct sum of identity permutations, with maximal elements to represent the places where descents occur. In other words, every element of the form $1 \oplus \rev(\Id_{k_1}) \oplus \dots \oplus \rev(\Id_{k_m})$ can be turned into an element of the form $((\Id_{k_1'} n \oplus \Id_{k_2'}) (n-1) \oplus \dots \oplus \Id_{k_m'}) (n-m+1)$ such that the place of every descent and ascent is preserved. And the same argument works in reverse, so we conclude that there is a bijection between $B_n(213,231)$ and $B_n(231,312)$.

We show a bijection between $B_n(231,312)$ and $B_n(132,312)$. 
Note that elements in $B_n(231,312)$ can be written in the form $1 \oplus \rev(\Id_{k_1}) \oplus \dots \oplus \rev(\Id_{k_m})$, where each $\Id_{k_i}$ may be one element. 
Observe that elements in $B_n(132,312)$ can be written in the form $(((( \cdots (m \oplus \Id_{k_{m}}) \ominus \cdots ) \ominus 2 ) \oplus \Id_{k_2} ) \ominus 1 ) \oplus \Id_{k_1}$, where $1, \dots , m$ are the first $m$ minimal elements in $\sigma$ and $\Id_{k_i}$ may be empty. 
These are also shown in Figure \ref{Wilfs}. Now we will turn $\sigma$ into an element of $B_n(231,312)$. Note that by this construction, there will always be a descent after each identity permutation in the sum, which we may write in terms of a reverse of an identity permutation. Also noting that $\Id_{k_i}$ may be written as $\rev(1) \oplus \dots \oplus \rev(1)$, we can turn the expression above to $1 \oplus \rev(\Id_{k_1}) \oplus \dots \oplus \rev(\Id_{k_j})$ while preserving every descent and ascent. A similar argument works in reverse, and hence there is a bijection between $B_n(132,312)$ and $B_n(231,312)$.
\end{proof}

\begin{figure}[htp]
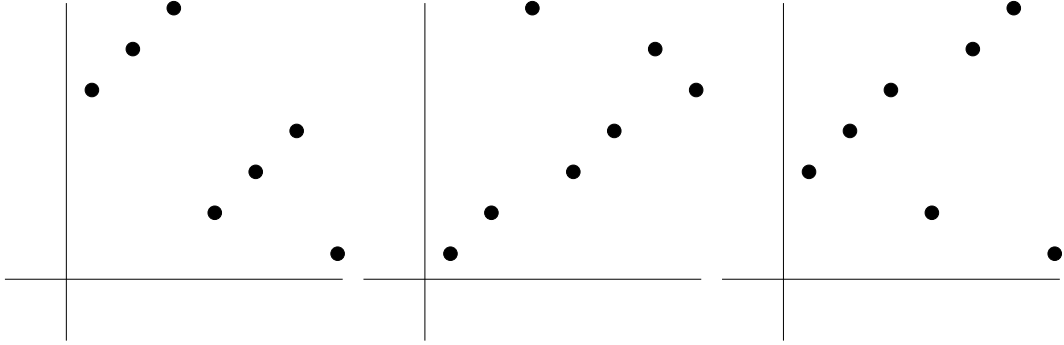

    \centering
    \begin{subfigure}[t]{.3\linewidth}
    \centering
    \resizebox{\columnwidth}{!}{%
    \drawperms{black/{1/5,2/6,3/7}/left,black/{4/2,5/3,6/4,7/1}/right}
    }
    \caption{The form of a permutation in $B_n(132,213)$.}
  \end{subfigure}
  \begin{subfigure}[t]{.3\linewidth}
    \centering
     \resizebox{\columnwidth}{!}{%
    \drawperms{black/{1/1,2/2,3/7}/left,black/{4/3,5/4,6/6,7/5}/right}
    }
    \caption{The form of a permutation in $B_n(213,231)$.}
  \end{subfigure}
  \begin{subfigure}[t]{.3\linewidth}
    \centering
    \resizebox{\columnwidth}{!}{%
    \drawperms{black/{1/3,2/4,3/5}/left,black/{4/2,5/6,6/7,7/1}/right}
    }
    \caption{The form of a permutation in $B_n(132,312)$.}
  \end{subfigure}
    \caption{Example forms of permutations in $B_n(132,213)$, $B_n(213,231)$, and $B_n(231,312)$. All of these permutations can be mapped to each other and to the left factor $UUDUUD$, as will be shown in Theorem \ref{132,213}.}
    \label{Wilfs}
\end{figure}

Note that this result also shows that the distribution of descents is consistent for all elements in these four sets.

We will show in the following theorem that the elements in $B_n(132,213)$ are in bijection with left factors of Dyck paths of $n-1$ steps. But first we provide the following definition:

\begin{definition}
A \emph{left factor} of a Dyck path is the path made up of all steps that precede the last $D$ step in a Dyck path. Left factors of Dyck paths of $n$ steps are left factors of all possible Dyck paths such that the path preceding the last $D$ step contains $n$ steps.
\end{definition}

\begin{example}\label{left factor}
Consider the Dyck path $UUDUUDDDUUDD$ shown in Figure \ref{Dyck}. The left factor associated with this Dyck path is $UUDUUDDDUU$.
\end{example}

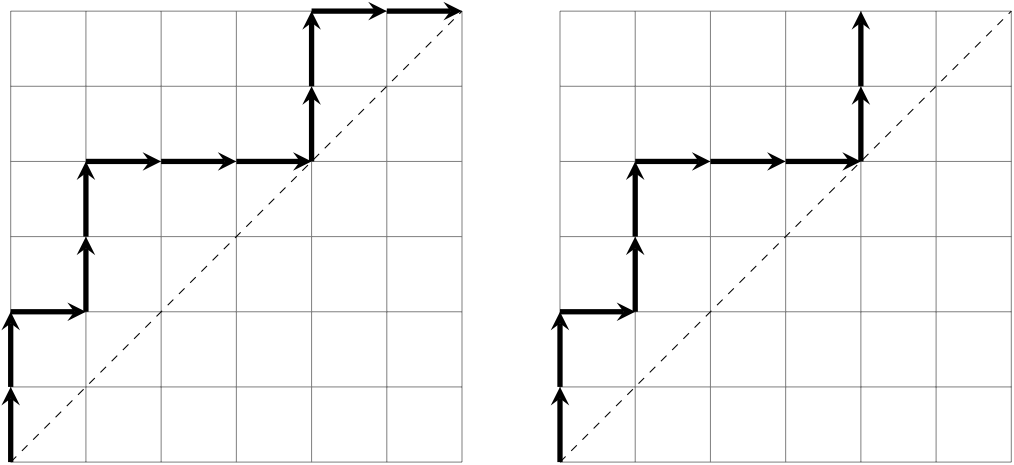
\begin{figure}[htp]
    \centering
    \begin{tikzpicture}
    \NEpath{0,0}{6}{6}{1,1,0,1,1,0,0,0,1,1,0,0};
    \end{tikzpicture}
    \hspace{1cm}
    \begin{tikzpicture}
    \NEpath{0,0}{6}{6}{1,1,0,1,1,0,0,0,1,1};
    \end{tikzpicture}
    \caption{The Dyck path and its corresponding left factor in Example \ref{left factor}.}
    \label{Dyck}
\end{figure}

The following theorem presents a bijection between ballot permutations avoiding $132$ and $213$ with left factors of Dyck paths. Since every prefix of a ballot permutation contains no more descents than ascents, this makes left factors of Dyck paths a very natural combinatorial object to biject to. 

\begin{theorem}\label{132,213}
The elements in $B_n(132,213)$ are in bijection with left factors of Dyck paths of $n-1$ steps, which are counted by the OEIS sequence \href{https://oeis.org/A001405}{A001405} \cite{oeis}.
\end{theorem}
\begin{proof}
Note that the elements in $B_n(132,213)$ are the skew sum of consecutively increasing permutations. Moreover, since they have to be ballot, the first two elements in any $\sigma \in B_n(132,213)$ must be increasing.

Let $\sigma = \Id_{k_1} \ominus \Id_{k_2} \ominus \dots \ominus \Id_{k_m}$. Note that $k_1 >1$ while any other $k_i$ may equal $1$. Now we can group together consecutive $k_i$s where each $k_i = 1$. So we get $\sigma = \Id_{k_1} \ominus \rev(\Id_{\ell_1}) \ominus \Id_{\ell_2} \ominus \cdots$. Then note that $\Id_{k_1} \ominus \rev(\Id_{\ell_1})$ uniquely determine a series of ups and downs in the left factor Dyck path. We can use the same argument for the rest of the terms in the direct sum of $\sigma$ to conclude that each $\sigma \in B_n(132,213)$ uniquely determine a series of ups and downs in a left factor Dyck path. And we can see that this argument works in reverse as well, since consecutive ascents can be grouped together into an identity term in $\sigma$, and consecutive descents can be grouped together into a reverse identity term in $\sigma$. Hence we conclude that there exists a bijection between the elements in $B_n(132,213)$ and left factors of Dyck paths of $n-1$ steps. 

And hence \begin{align*}
    |B_{n+1}(132,213)| = {n \choose {\lfloor \frac{n}{2} \rfloor}}. & \qedhere
\end{align*}
\end{proof}

The following example illustrates the bijection presented in Theorem \ref{132,213}.

\begin{example}\label{bijectexample}
The ballot permutation $\sigma = 456312$ is in $B(132,213)$ and is in bijection with the left factor $UUDDU$, as shown in Figure \ref{biject}.
\end{example}

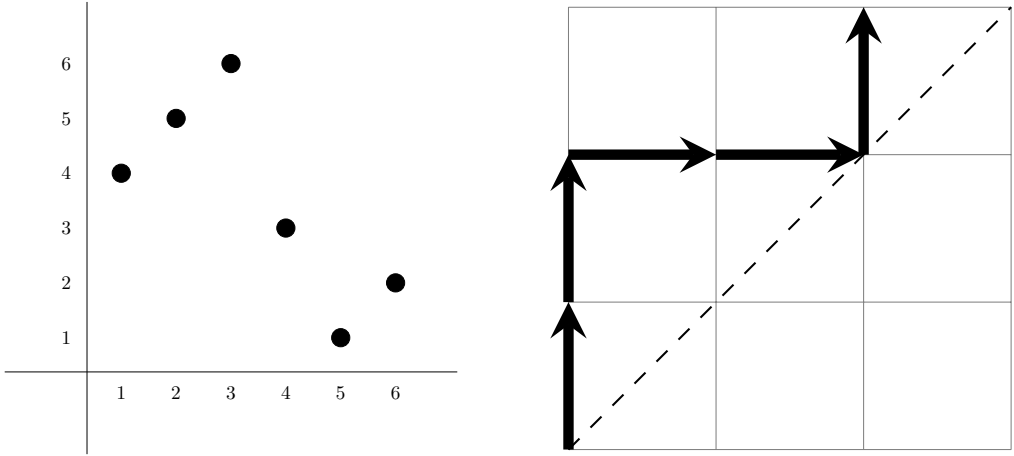
\begin{figure}[htp]
    \centering
    \resizebox{2\columnwidth/5}{!}{%
    \drawpermutate{black/{1/4,3/6,5/1}/left,black/{4/3,2/5,6/2}/right}
    }
    \hspace{1cm}
    \resizebox{2\columnwidth/5}{!}{%
    \begin{tikzpicture}
    \NEpath{0,0}{3}{3}{1,1,0,0,1};
    \end{tikzpicture}
    }
    \caption{A permutation in $B_6(132,213)$ and its corresponding left factor of a Dyck path in Example \ref{bijectexample}.}
    \label{biject}
\end{figure}

\begin{theorem}\label{132,231}
There exists a unique ballot permutation avoiding the patterns $132$ and $231$.
\end{theorem}
\begin{proof}
Let $\sigma \in B_n(132,231)$. Then the first two elements of $\sigma$ must be increasing since $\sigma$ is a ballot permutation. Moreover, these two elements must also be consecutive to avoid an occurrence of $132$. So call these two elements $k$ and $k+1$. Elements smaller than $k$ must be placed before $k$ to avoid an occurrence of $231$. Starting from the minimal element $1$, elements less than $k$ must be placed consecutively to avoid $132$, so we conclude that $\sigma = \Id_n$.
\end{proof}

Now we show a bijection between ballot permutations of length $n+1$ avoiding the patterns $132$ and $321$ with permutations of length $n$ avoiding the same patterns. This involves removing the second element of the ballot permutation and noting that the remaining subpermutation will still avoid $132$ and $321$.

\begin{theorem}\label{132,321}
The elements in $B_{n+1}(132,321)$ are in bijection with the elements in $S_n(132,321)$.
\end{theorem}
\begin{proof}
Let $\sigma \in B_{n+1}(132,321)$. Then since $\sigma$ is a ballot permutation, the first two elements of $\sigma$ are increasing, and they must be consecutive to avoid an occurrence of $132$. So let us write $\sigma = k(k+1) \sigma_R$. Then removing an element will still avoid these patterns, so $k \sigma_R \in S_n(132,321)$.

Now let $\sigma \in S_n(132,321)$. Then write $\sigma = \sigma_L n \sigma_R$, where $\sigma_L n$ and $\sigma_R$ are both consecutively increasing. Note that $\sigma$ has at most one descent. 

Then inserting a consecutive increasing element into the second index of $\sigma$ and standardizing everything else preserves the number of descents, and moreover, still avoids the patterns $132$ and $321$. More specifically, $\sigma_L n (n+1) \sigma_R$ will still avoid $132$ and $321$ (this is the inverse of the map above since $\sigma_L n$ is consecutively increasing). This permutation will still be a ballot permutation since it starts with an ascent and the permutation contains at most one descent, because $\sigma_L$ and $\sigma_R$ are both consecutively increasing. 

And hence $\sigma_L n (n+1) \sigma_R \in B_{n+1}(132,321)$. This is sufficient to show a bijection between the elements in $B_{n+1}(132,321)$ and the elements in $S_n(132,321)$. Simion and Schmidt \cite{simion1985restricted} proved that $|S_n(132,321)| = {n \choose 2} + 1$, so 

\begin{align*}
    |B_{n+1}(132,321)| = |S_n(132,321)| ={n \choose 2} + 1. & \qedhere
\end{align*}
\end{proof}

\begin{theorem}\label{213,312}
Let $a_n = |B_n(213,312)|$. Then $$a_n = \sum_{k=0}^{\lfloor \frac{n}{2} \rfloor} {n \choose k},$$ which is listed as the OEIS sequence \href{https://oeis.org/A027306}{A027306} \cite{oeis}.
\end{theorem}
\begin{proof}
Let $\sigma \in B_n(213,312)$. Then writing $\sigma = \sigma_L n \sigma_R$, note that $\sigma_L$ must be increasing and $\sigma_R$ must be decreasing (but not necessarily consecutive). 

Then we construct all possible ballot permutations in $B_n(213,312)$. Let $|\sigma_R| = k$. Then there are $n \choose k$ ways to pick the elements in $\sigma_R$, which are forced to decrease. The rest of the elements must be in $\sigma_L$, which are forced to increase. Hence there are $n \choose k$ ways to construct $\sigma$. But $k \leq \lfloor \frac{n}{2} \rfloor$, since we cannot have more descents than ascents. And hence \begin{align*}
    a_n = \sum_{k=0}^{\lfloor \frac{n}{2} \rfloor} {n \choose k}. & \qedhere
\end{align*}
\end{proof}

Next we can show a bijection between ballot permutations avoiding $213$ and $321$ with permutations avoiding the same patterns that are not of the form $n \Id_{n-1}$, which is clearly not a ballot permutation.

\begin{theorem}\label{213,321}
Elements in $B_n(213,321)$ are in bijection with elements in $S_n(213,321) \setminus (n \Id_{n-1})$.
\end{theorem}
\begin{proof}
It's clear that every $\sigma \in B_n(213,321)$ is in $S_n(213,321) \setminus (n \Id_{n-1})$. 

Now let $\sigma \in S_n(213,321)$. Then let us write $\sigma = \sigma_L n \sigma_R$, noting $\sigma_L$ and $\sigma_R$ are both increasing to avoid $213$ and $321$. This means that $\sigma$ has at most one descent. Note that if $\sigma_L$ is nonempty, then $\sigma$ is a ballot permutation that avoids $213$ and $321$. But there is only one permutation in $S_n(213,321)$ where $\sigma_L$ is empty: namely, $n \Id_{n-1}$. And hence we conclude that every $\sigma \in S_n(213,321)$ that is not  $n \Id_{n-1}$ is in $B_n(213,321)$, and hence there is a bijection between the elements in $B_n(213,321)$ and the elements in $S_n(213,321) \setminus (n \Id_{n-1})$.

So by Simion and Schmidt \cite{simion1985restricted}, we conclude that \begin{align*}
    |B_n(213,321)| = |S_n(213,321)| - 1 = {n \choose{2}}. & \qedhere
\end{align*}
\end{proof}

The following theorem shows a bijection between $B_{n+1}(231,321)$ and $S_n(231,321)$, which involves removing the first element of every ballot permutation avoiding $231$ and $321$ and noting that the remaining permutation will still avoid these patterns.

\begin{theorem}\label{231,321}
The elements in $B_{n+1}(231,321)$ are in bijection with the elements in $S_n(231,321)$.
\end{theorem}
\begin{proof}
Let $\sigma \in B_{n+1}(231,321)$. Note that the minimum element $1$ must be either the first element or the second element of $\sigma$ to avoid $231$ and $321$. However, since $\sigma$ is a ballot permutation, it cannot start with a descent, and hence $1$ must be the first element. Note that removing $1$ from $\sigma$ will still avoid $231$ and $321$, and hence is an element in $S_n(231,321)$.

Now let $\sigma \in S_n(231,321)$. Since $\sigma$ avoids $321$, there are no consecutive descents, which means that there is at most one more descent than ascent. Then note that if we insert a minimal element $0$ at the beginning of $\sigma$, then $0\sigma$ will still $231$ and $321$. Moreover, we've guaranteed one ascent at the beginning of the permutation, and there are still no consecutive descents. Hence there at at least as many ascents as descents in $\sigma$, and hence $\sigma \in B_{n+1}(231,321)$.

This is sufficient to show a bijection between the elements in $B_{n+1}(231,321)$ and the elements in $S_n(231,321)$. By Simion and Schmidt \cite{simion1985restricted}, we conclude that \begin{align*}
    |B_{n+1}(231,321)| = |S_n(231,321)| = 2^{n-1}. & \qedhere
\end{align*}
\end{proof}

And lastly, we provide a constructive approach to show the following result: 

\begin{theorem}\label{312,321}
Let $a_n = |B_n(312,321)|$. Then $a_n = 3 \cdot 2^{n-3}$ for $a_n \geq 3$. 
\end{theorem}
\begin{proof}
Let $\sigma \in B_n(312,321)$ and write $\sigma = \sigma_L r$, where $r \in [n]$. We insert a maximal element $(n+1)$ into $\sigma$ to generate an element in $B_{n+1}(312,321)$. Note that we must insert $(n+1)$ adjacent to $r$ in $\sigma$ to avoid an occurrence of $312$ and $321$. Further, $\sigma_L (n+1) r$ avoids $312$ and $321$ and is further ballot. A similar argument shows that $\sigma_L r (n+1)$ is also in $B_{n+1}(312,321)$. So each $\sigma \in B_n(312,321)$ will generate two distinct elements in $B_{n+1}(312,321)$.

Since we've shown that $(n+1)$ must be inserted adjacent to the last element of a permutation in $B_n(312,321)$, now we show that inserting $(n+1)$ anywhere else into some $\sigma' \notin B_n(312,321)$ will not generate an element in $B_{n+1}(312,321)$.

As stated above, we must insert $(n+1)$ adjacent to the last element of $\sigma'$ to avoid $312$ and $321$. Now write $\sigma' = \sigma_L' r'$. We have two permutations to consider:

\begin{enumerate}
    \item $\sigma_L' r' (n+1)$
    
    Then note that $\sigma_L' r'$ must be ballot and avoid $312$ and $321$ as well, which is impossible.
    
    \item $\sigma_L' (n+1) r'$
    
    Now if $\sigma_L' r'$ does not avoid $312$ and $321$, then $\sigma_L' (n+1) r'$ does not avoid these patterns either. If $\sigma_L' r'$ does avoid $312$ and $321$, then it must not be a ballot permutation. But since this permutation avoids $321$, there cannot exist consecutive descents in this permutation. Note that $\sigma_L' r'$ cannot start with an ascent, or else it would be a ballot permutation. And hence $\sigma_L' r'$ must start with a descent, and hence $\sigma_L' (n+1) r'$ is not a ballot permutation.
\end{enumerate}

So inserting a maximal element $(n+1)$ anywhere else of $\sigma' \notin B_n(312,321)$ will not generate an element in $B_{n+1}(312,321)$. So we conclude that $a_{n+1} = 2a_n$. Since we know that $a_3 = 3$, then we conclude that $a_n = 3 \cdot 2^{n-3}$.
\end{proof}

\section{Enumeration of Pattern Avoidance Classes of length 3}

Having enumerated all $3$-permutations avoiding double restrictions, we now turn our attention to enumerating $3$-permutations avoiding triple restrictions, as Simion and Schmidt \cite{simion1985restricted} have done with classic permutations. Table \ref{triple avoidance} presents the sequence of ballot permutations avoiding three patterns of length $3$.

\begin{table}[h]
    \centering
    \begin{tabular}{|c | c | c | c | c|} 
 \hline
 Patterns & Sequence & OEIS Sequence & Comment \\ [0.5ex] 
 \hline\hline
 $123,132,213$ &  $1,1,1,1,1,1, \dots$ & & Sequence of all $1$s; Corollary \ref{123,132,213}  \\ 
 \hline
 $123,132,231$ &  $1,1,0,0,0,0, \dots$ & & Terminates after $n=2$ \\
 \hline
 $123,132,312$ & $1,1,1,0,0,0, \dots$ & & Terminates after $n=3$ \\
 \hline
 $123,132,321$ & $1,1,1,0,0,0, \dots$ & & Terminates after $n=3$ \\
 \hline
 $123,213,231$ & $1,1,1,0,0,0, \dots$ & & Terminates after $n=3$ \\
 \hline
 $123,213,312$ & $1,1,2,0,0,0, \dots$ & & Terminates after $n=3$ \\
 \hline
 $123,213,321$ & $1,1,2,1,0,0, \dots$ & & Terminates after $n=4$ \\
 \hline
 $123,231,312$ & $1,1,1,0,0,0, \dots$  & & Terminates after $n=3$ \\
 \hline
 $123, 231, 321$ & $1,1,1,0,0,0, \dots$ & & Terminates after $n=3$ \\
 \hline
 $123,312,321$ & $1,1,2,0,0,0, \dots$ & & Terminates after $n=3$ \\
 \hline
 $132,213,231$ & $1,1,1,1,1,1, \dots$ & & Sequence of all $1$s; Corollary \ref{triple corollary} \\
 \hline
 $132,213,312$ & $1,1,2,2,3,3, \dots$ & \href{https://oeis.org/A004526}{A004526} & Theorem \ref{132,213,312} \\
 \hline
 $132,213,321$ & $1,1,2,3,4,5, \dots$ & \href{https://oeis.org/A000027}{A000027} & Excluding $n=1$; Theorem \ref{132,213,321} \\
 \hline
 $132,231,312$ & $1,1,1,1,1,1, \dots$ & & Sequence of all $1$s; Corollary \ref{triple corollary} \\
 \hline
 $132,231,321$ & $1,1,1,1,1,1, \dots$ & & Sequence of all $1$s; Corollary \ref{triple corollary} \\
 \hline
 $132,312,321$ & $1,1,2,3,4,5, \dots$ & \href{https://oeis.org/A000027}{A000027}  & Excluding $n=1$; Theorem \ref{triple bijection} \\
 \hline
 $213,231,312$ & $1,1,2,2,3,3, \dots$ & \href{https://oeis.org/A004526}{A004526} & Theorem \ref{213,231,312} \\
 \hline
 $213,231,321$ & $1,1,2,3,4,5, \dots$ & \href{https://oeis.org/A000027}{A000027}  & Excluding $n=1$; Theorem \ref{triple bijection} \\
 \hline
 $213,312,321$ & $1,1,3,4,5,6, \dots$ & \href{https://oeis.org/A000027}{A000027}  & Excluding $n=1$ and $n=2$; Theorem \ref{213,312,321} \\
 \hline
 $231,312,321$ & $1,1,2,3,5,8, \dots$ & \href{https://oeis.org/A000045}{A000045} & Theorem \ref{231,312,321} \\
 \hline
\end{tabular}
    \caption{Sequences of ballot permutations avoiding three permutations of length $3$.}
    \label{triple avoidance}
\end{table}

\begin{corollary}\label{123,132,213}
We have $|B_n(123,132,213)|=1$ for all $n$.
\end{corollary}
\begin{proof}
This follows immediately from Theorem \ref{123,132}, since the unique permutation avoiding $123$ and $132$ also avoids $213$. Let $\sigma \in B_n(123,132,213)$. Specifically, $\sigma = (12) \ominus (12) \ominus \dots \ominus (12)$ when $n$ is even and $\sigma = (12) \ominus (12) \ominus \dots \ominus (12) \ominus (1)$ when $n$ is odd.
\end{proof}

\begin{corollary}\label{triple corollary}
We have $|B_n(132,213,231)| = |B_n(132,231,312)| = |B_n(132,231,321)| =1$ for all $n$.
\end{corollary}
\begin{proof}
This follows immediately from Theorem \ref{132,231}. In particular, \begin{align*}
    B_n(132,213,231) = B_n(132,231,312) = B_n(132,231,321) = \{ \Id_n \}. & \qedhere
\end{align*}
\end{proof}

Now we will show that the sets of patterns $\{132,213,312 \}$ and $\{213,231,312\}$ are Wilf-equivalent.

\begin{theorem}\label{213,231,312}
The sets $B_n(132,213,312)$ and $B_n(213,231,312)$ are Wilf-equivalent.
\end{theorem}
\begin{proof}
Note that an element in $B_n(132,213,312)$ can be written as $\Id_{k_L} \ominus \rev(\Id_{k_R})$. Similarly, an element in $B_n(213,231,312)$ can be written as $\Id_{k_L} \oplus \rev(\Id_{k_R})$.

Now we can write $\Id_{k_L} \oplus \rev(\Id_{k_R})$ as $\Id_{k_L} \oplus (1 \ominus \rev(\Id_{k_R-1}))$. But note that we can rewrite this as $\Id_{k_L+1} \ominus \rev(\Id_{k_R-1})$ while preserving the positions of every descent and ascent in the permutation. A similar reasoning applies for the reverse case, and hence there is a descent-preserving bijection between $B_n(132,213,312)$ and $B_n(213,231,312)$.
\end{proof}

\begin{theorem}\label{132,213,312}
Let $a_n = |B_n(132,213,312)|$. Then $a_n = \lfloor \frac{n+1}{2} \rfloor$.
\end{theorem}
\begin{proof}
Let $\sigma \in B_n(132,213,312)$. Then, because $\sigma$ avoids $132$, $213$, and $312$, it can be written in the form $\sigma_L n \sigma_R$, where $\sigma_L n$ is consecutively increasing and $\sigma_R$ is consecutively decreasing. Now we count how many different $\sigma$ there are. Note that $n$ can be in the last $\lfloor \frac{n+1}{2} \rfloor$ places to ensure that there are at least as many ascents as descents in $\sigma$. And hence $|B_n(132,213,312)| = \lfloor \frac{n+1}{2} \rfloor$.
\end{proof}

Now we show a Wilf-equivalence between three other sets of patterns.

\begin{theorem}\label{triple bijection}
The three sets $B_n(132,213,321)$, $B_n(132,312,321)$, and $B_n(213,231,321)$ are Wilf-equivalent.
\end{theorem}
\begin{proof}
Note that an element in $B_n(132,213,321)$ can be written as $\Id_{k_L} \ominus (1 \oplus \Id_{k_R})$. Similarly, an element in $B_n(132,312,321)$ can be written as $(\Id_{k_L} \ominus 1) \oplus \Id_{k_R}$ and an element in $B_n(213,231,321)$ can be written as $\Id_{k_L} \oplus (1 \ominus \Id_{k_R})$.

Observe that for each value of $k_L, k_R \in \mathbb N$ with $k_L+k_R+1 = n$, we can send $\Id_{k_L} \ominus (1 \oplus \Id_{k_R})$ to $(\Id_{k_L} \ominus 1) \oplus \Id_{k_R}$. This preserves the position of every descent in the permutation and gives our bijection.

Similarly,  $\Id_{k_L} \ominus (1 \oplus \Id_{k_R})$ can be bijected to $\Id_{k_L-1} \oplus (1 \ominus \Id_{k_R+1})$, which also preserves the position of every descent in the permutation.

Since the bijections from $S_n(132,213,321)$ to $S_n(132,312,321)$ to $S_n(213,231,321)$ are descent-preserving, we can now restrict them to bijections from $B_n(132,213,321)$ to $B_n(132,312,321)$ to $B_n(213,231,321)$.

Hence there exists descent-preserving bijections between $B_n(132,213,321)$ and $B_n(132,312,321)$ and between $B_n(132,213,321)$ and $B_n(213,231,321)$, and all three sets are Wilf-equivalent.
\end{proof}

\begin{theorem}\label{132,213,321}
Let $a_n = |B_n(132,213,321)|$. Then $a_n = n-1$.
\end{theorem}
\begin{proof}
Let $\sigma \in B_n(132,213,321)$. Then, because $\sigma$ avoids $132$, $213$, and $321$, then it can be written in the form of $\sigma_L n \sigma_R$, where $\sigma_L n$ is consecutively increasing and $\sigma_R$ is consecutively increasing. Note that there is at most one descent in this permutation, and hence $n$ can be anywhere except the first element for $\sigma$ to be a ballot permutation (in other words, $\sigma_L$ cannot be empty), and hence there are $n-1$ different permutations in $B_n(132,213,321)$.
\end{proof}

\begin{theorem}\label{213,312,321}
Let $a_n = |B_n(213,312,321)|$. Then $a_n = n$.
\end{theorem}
\begin{proof}
Let $\sigma$ be in $B_n(213,312,321)$. Now $\sigma$ can be written as $\sigma_L n \sigma_R$, where $\sigma_L$ is increasing and $\sigma_R$ is either empty or one element to avoid $312$ and $321$.

When $\sigma_R$ is empty, the identity permutation is the only one that satisfies the above criteria. When $\sigma_R$ is nonempty, we can choose $n-1$ different elements to be the last element. Then all the other elements must go in increasing order in $\sigma_L$, so there are a total of $n$ different permutations in $B_n(213,312,321)$.
\end{proof}

Finally we present a constructive approach to show that ballot permutations avoiding the patterns $231$, $312$, and $321$ follow the Fibonacci sequence with initial terms $a_1 = 1$ and $a_2 = 1$.

\begin{theorem}\label{231,312,321}
Let $a_n = |B_n(231,312,321)|$. Then $a_n$ follows the recurrence relation $a_{n} = a_{n-1} + a_{n-2}$ with the initial terms $a_1 = 1$ and $a_2 = 1$, which is the Fibonacci sequence.
\end{theorem}
\begin{proof}
Note that given some $\sigma \in B_{n-1}(231,312,321)$, the permutation $\sigma n$ will be in $B_{n}(231,312,321)$, since inserting $n$ at the end of a permutation that avoids $231,312$, and $321$ will still avoid these three permutations. Moreover, an ascent has been added by inserting $n$ onto the end of $\sigma$, and hence $\sigma n$ will still be a ballot permutation. This case contributes $a_{n-1}$ different elements in $B_{n}(231,312,321)$.

Now also note that given some $\tau \in B_{n-2}(231,312,321)$, the permutation $\tau n (n-1)$ will also be in $B_{n}(231,312,321)$. This still avoids $231$, $312$, and $321$, and we've added an ascent followed by a descent, so $\tau n (n-1)$ is still a ballot permutation. This case contributes $a_{n-2}$ different elements in $B_{n}(231,312,321)$.

Given $\sigma \in B_{n-1}(231,312,321)$, we show that inserting the maximal element $n$ in any other place cannot produce an element in $B_{n}(231,312,321)$. Now if $\sigma$ ends in $n-1$ and we insert $n$ left-adjacent to $n-1$, this case is already accounted for above because this is in the form of $\tau n(n-1)$, where $\tau \in B_{n-2}(231,312,321)$. If $\sigma$ ends in $k<n-1$, then inserting $n$ left-adjacent to $k$ will contain an occurrence of $231$. Inserting $n$ anywhere else will contain either an occurrence of $321$ or $312$, since these cases are disjoint.

For $\sigma \notin B_{n-1}(231,312,321)$, we show that we cannot produce an element in $B_{n}(231,312,321)$ by inserting the maximal element $n$ anywhere. Note that if $\sigma$ contains either $231$, $312$, or $321$, inserting $n$ anywhere will still contain an occurrence of these patterns. Now let $\sigma$ be a non-ballot permutation. Note that we must insert $n$ adjacent to the last element of $\sigma$ or else there is an occurrence of $312$ or $321$. If $n$ is inserted left-adjacent to the last element, then $\sigma$ must be $\sigma_L (n-1)$ to avoid $231$. Then $\sigma_L n (n-1)$ is not a ballot permutation because we've inserted a descent at the end of $\sigma_L (n-1)$. Now if we insert $n$ at the end of $\sigma$, note that $\sigma$ is a prefix of $\sigma n$. And since $\sigma$ is not a ballot permutation,  $\sigma n$ cannot be either. And hence if we insert $n$ anywhere else, we cannot produce an element in $B_{n}(231,312,321)$.

Now let $\tau \in B_{n-2}(231,312,321)$. We show that we cannot produce an element in $B_{n}(231,312,321)$ by inserting the maximal elements $n-1$ and $n$ in any other places. Now note that $\tau (n-1) \in B_{n-1}(231,312,321)$, which is covered by the other case above. Now similar to the reasoning above, we have to insert $n-1$ left-adjacent to the last element of $\tau$. And doing this forces the last element of $\tau$ to be $n-2$ to avoid $231$. So write $\tau$ as $\tau_L (n-2)$ and consider $\tau_L (n-1) (n-2)$. Then note that $\tau_L (n-1) (n-2)n$ is already counted in the case above since $\tau_L (n-1) (n-2) \in B_{n-1}(231,312,321)$. Moreover, $\tau_L (n-1) n (n-2)$ contains $231$, so inserting $n$ and $n-1$ anywhere else in $\tau$ will not produce an element in $B_{n}(231,312,321)$.

For $\tau \notin B_{n-2}(231,312,321)$, we show that we cannot produce an element in $B_{n}(231,312,321)$ by inserting the maximal elements $n$ and $n-1$ anywhere. As discussed above, if $\tau$ contains $231$, $312$, or $321$, then inserting $n$ and $n-1$ anywhere in $\tau$ will still contain these patterns. So assume that $\tau$ is not a ballot permutation. So we must insert $n-1$ either left-adjacent or right-adjacent to the last element of $\tau$. 

Let us consider the case where we insert $n-1$ left-adjacent to the last element of $\tau$. Similarly as above, this forces the last element of $\tau$ to be $n-2$ to avoid $231$. So write $\tau$ as $\tau_L (n-2)$ and consider $\tau_L (n-1) (n-2)$. Now this is simply adding a descent at the end of $\tau_L (n-2)$. Similarly, we must insert $n$ adjacent to $n-2$, and hence $\tau_L (n-1) (n-2)$ is not a ballot permutation. This implies that $\tau_L (n-1) (n-2) n$ is also not a ballot permutation. Moreover, $\tau_L (n-1) n (n-2)$ contains an occurrence of $231$.

Note that if we insert $n-1$ right-adjacent to the last element of $\tau$, then $\tau (n-1)$ is not a ballot permutation because $\tau$ is a prefix of $\tau (n-1)$ and $\tau$ is not a ballot permutation. Hence both $\tau (n-1) n$ and $\tau n (n-1)$ cannot be ballot permutations.

Hence if we insert $n$ and $n-1$ anywhere else, we cannot produce an element in $B_{n}(231,312,321)$.

So we conclude that \begin{align*}
    a_{n} = a_{n-1} + a_{n-2}. & \qedhere
\end{align*}
\end{proof}

\section{Conclusion and Open Problems}

In this paper, we have exhaustively enumerated ballot permutations avoiding two patterns of length $3$ and three patterns of length $3$. The results presented in this paper extend Lin, Wang, and Zhao's \cite{lin2022decomposition} enumeration of permutations avoiding a single pattern of length $3$ and proved Wilf-equivalences of pattern classes. In particular, bijections between ballot permutations avoiding certain patterns and left factors of Dyck paths were also shown. We conclude with the following open problems as proposed by Lin, Wang, and Zhao's \cite{lin2022decomposition}:

\begin{problem}
Can ballot permutations avoiding sets of patterns of length $4$ be enumerated?
\end{problem}

Although this paper has shown connections between ballot permutations avoiding patterns of length $3$ and their recurrence relations and formulas, there are no existing OEIS sequences \cite{oeis} that correspond with the number of ballot permutations avoiding one pattern with length $4$. Moreover, can ballot permutations avoiding consecutive patterns or vincular patterns be enumerated?

And finally, Lin, Wang, and Zhao \cite{lin2022decomposition} have suggested the following notion of a \emph{ballot multipermutation}.
\begin{definition}
For a tuple of natural numbers $\textbf{m} = (m_1, \dots, m_n)$, let $\mathfrak{S}_{\textbf{m}}$ be the set of multipermutations of $\{ 1^{m_1}, 2^{m_2}, \dots, n^{m_n} \}$. An element $\sigma \in \mathfrak{S}$ is a ballot multipermutation if for each $i$ such that $1 \leq i \leq \sum_{k=1}^n m_k$, the following inequality holds: 
$$|\{ j \in [i] : \sigma(j) < \sigma(j+1)  \}| \geq |\{ j \in [i] : \sigma(j) > \sigma(j+1)  \}|.$$
\end{definition}

\begin{problem}
For fixed $\textbf{m}$, is it possible to enumerate ballot multipermutations in $\mathfrak{S}_{\textbf{m}}$? Further, is it possible to enumerate ballot multipermutations avoiding patterns in $S_n$?
\end{problem}

In addition, inspired by Bertrand's \cite{bertrand} ballot problem for $\lambda > 1$, we propose the following problem:

\begin{problem}
Can ballot permutations avoiding a single pattern of length $3$ or pairs of patterns of length $3$ with at least $\lambda$ times as many ascents as descents be enumerated?
\end{problem}

Furthermore, the enumeration of even and odd ballot permutations avoiding small patterns has not been studied and would be a further avenue for future research.

\section*{Acknowledgements}
This research was conducted at the 2022 University of Minnesota Duluth REU and is supported by Jane Street Capital, the NSA (grant number H98230-22-1-0015), the NSF (grant number DMS2052036), and the Harvard College Research Program. The author is indebted to Joe Gallian for his dedication and organizing the University of Minnesota Duluth REU. Lastly, a special thanks to Joe Gallian, Amanda Burcroff, Michael Ren, and Katalin Berlow for their invaluable feedback and advice on this paper.

\newpage
\bibliography{references}
\bibliographystyle{plain}

$\\ \\ \\$
\end{document}